\documentclass[12pt, a4paper]{amsart}

\usepackage{amsmath}
\usepackage{amssymb}
\usepackage{amscd}
\usepackage{latexsym}
\usepackage{amsthm}
\usepackage[all]{xy}

\input xypic

\newtheorem{thm}[equation]{Theorem}
\newtheorem{cor}[equation]{Corollary}
\newtheorem{lem}[equation]{Lemma}

\theoremstyle{definition}

\newtheorem{rem}[equation]{Remark}
\newtheorem{exa}[equation]{Example}

\numberwithin{equation}{section}

\newcommand{\Hom}{\operatorname{Hom}}

\newcommand{\letbe}{\mathbin{\raisebox{.3pt}{:}\!=}}

\newcommand{\colim}{\operatorname{colim}}

\newcommand{\bC}{\mathbb{C}}

\newcommand{\bQ}{\mathbb{Q}}
\newcommand{\bR}{\mathbb{R}}
\newcommand{\bZ}{\mathbb{Z}}

\newcommand{\srf}[1]{\mbox{${\text{\it SR}^F}$}}

\newcommand{\djs}{\mbox{\it DJ\/}}

\newcommand{\cat}[1]{\mbox{\sc #1}}

\newcommand\Z{\bZ}
\newcommand\z{\Z}
\newcommand\q{\bQ}
\newcommand\Q{\bQ}
\newcommand\C{\bC}
\newcommand\R{\bR}

\newcommand \lra{\longrightarrow}

\newcommand \da{\downarrow}
\def \larrow#1{\,\stackrel{#1}\lra\,}

\newcommand{\catk}{\cat{cat}(K)}

\newcommand{\Top}{\cat{Top}}
\newcommand{\nz}{\newline}
\newcommand{\quer}{\overline}

\newcommand{\0}{\emptyset}

\begin{document}
\bibliographystyle{plain}

\title{Colorings of simplicial complexes and vector
bundles over
Davis-Januszkiewicz  spaces}

\author{Dietrich Notbohm}
\address{
Department of Mathematics,
Vrije Universiteit,
Faculty of Sciences,
De Boelelaan 1081a,
1081 HV Amsterdam,
The Netherlands
}
\email{notbohm@few.vu.nl}



\begin{abstract}
We show that coloring properties of a simplicial complex $K$
are reflected by splitting properties of a bundle over the associated
Davis-Januszkiewicz space whose Chern classes are given by
the elementary symmetric polynomials in the generators of the Stanley-Reisner algebra
of $K$.
\end{abstract}

\maketitle

\markright{COLORINGS AND VECTOR BUNDLES} 
%
%
%
%
%
%
%
%
%

\section{Introduction}\label{introduction}

 For a simplicial complex $K$, Davis and Januszkiewicz constructed a family
of spaces, all of which are homotopy equivalent, and
whose integral cohomology is isomorphic to the associated
Stanley-Reisner algebra $\Z[K]$ \cite[Section 4]{daja}. We denote a  generic
model for this homotopy type by
$\djs(K)$. In the above mentioned influential paper, Davis and Januszkiewicz
also constructed a particular complex vector bundle
$\lambda$ over $\djs(K)$ whose Chern classes are given by
the elementary symmetric
polynomials in the generators of $\z[K]$ \cite[Section 6]{daja}. This vector bundle
is of particular interest. For example, if
$K$ is the dual of the boundary of a simple polytope $P$, then
the associated moment angle
complex $Z_K$ is a manifold
and the realification $\lambda_\bR$ of $\lambda$
is stably isomorphic to the bundle given by applying the Borel
construction to the tangent bundle of $Z_K$.
And if $M^{2n}$ is a quasitoric  manifold over $P$,
then again the
Borel construction applied to the tangent bundle of $M^{2n}$ produces a
vector bundle stably isomorphic to $\lambda_\R$  \cite[Theorem 6.6, Lemma 6.5]{daja}.

Davis and Januszkiewicz also noticed that, if $K$ is the dual of the boundary of a polytope
of dimension $n$ and
admits a coloring with $n$ colors,
the bundle $\lambda$ splits into a direct sum of $n$ complex line bundles  and a trivial bundle
\cite[Section 6.2]{daja}.
We are interested in generalizations of this observation. In fact, we will show that
a simplicial complex admits a coloring with $r$ colors precisely when $\lambda$ splits stably
into a direct sum of $r$ linear complex bundles and a trivial bundle.
We will also show that a similar result holds for the realification of $\lambda$.

To make our statements more precise we have to fix notation and
recall some basic constructions.
Let $[m]\letbe \{1,...,m\}$ be the set of the first $m$ natural numbers.
A finite abstract simplicial complex $K$ on $[m]$ is given by
a set of faces $\alpha\subseteq [m]$
which is closed under the formation of subsets.
We consider the empty set $\0$ as a face of $K$.
The dimension $\dim \alpha$ of a face $\alpha$
is given in terms of its cardinality by $|\alpha| -1$,
and the dimension $\dim K$ of
$K$ is the maximum of the dimensions of its faces.
The most basic examples are given by
full simplices. For $\alpha \subseteq [m]$ we denote by $\Delta[\alpha]$
the simplicial complex which consists of all possible subsets of $\alpha$.
Then $\Delta[\alpha]$ is an $(|\alpha|-1)$-dimensional simplex.
The full simplex $\Delta[m]$ contains $K$ as a subcomplex, and if $\sigma\in K$ then
$\Delta(\sigma)\subset K$ is a subcomplex as well.

A {\it regular $r$-paint coloring}, an {\it $r$-coloring} for short,
of a simplicial complex $K$ is a non degenerate simplicial map
$g\colon K\lra \Delta [r]$, i.e. $g$ maps each face of $K$ isomorphically on a
face of $\Delta [r]$. The inclusion $K\subset \Delta[m]$ always provides an m-coloring.
If $\dim(K)=n-1$, then $K$ may only allow $r$-colorings for $r\geq n$.

For a commutative ring $R$ with unit we denote by $R[m]\letbe R[v_1,...,v_m]$
the
graded polynomial algebra generated by the algebraically independent
elements $v_1,...,v_m$ of degree 2, one for each vertex of $K$.
For each subset $\alpha\subseteq [m]$ we denote by
$v_\alpha\letbe \prod_{j\in \alpha} v_j$ the square free monomial whose factors
are in 1 to 1 correspondence with vertices contained in $\alpha$.
The graded
Stanley-Reisner algebra $R[K]$ associated with $K$
is defined as the quotient $R[K]\letbe R[m]/I_K$,
where  $I_K\subset R[m]$ is the ideal generated
by all elements $v_\mu$ such that $\mu\subseteq [m]$ is not a face of $K$.

Since $BT^m$ is an Eilenberg-MacLane space realizing the polynomial algebra
$\Z[m]$, the projection $\Z[m] \lra \Z[K]$ can be realized by a map
$f:DJ(K) \lra BT^m$. We can think of $T^m$ as the maximal torus of the unitary group $U(m)$.
The pull back along the composition $DJ(K) \lra BT^m \lra BU(m)$ of the universal bundle over
$BU(m)$ gives a vector bundle $\lambda \da  DJ(K)$.
This is the vector bundle studied by Davis and Januszkiewicz and mentioned above.
The total Chern class $c(\lambda)=1+ \sum_i c_i(\lambda)$ of $\lambda$ is then
given by $c(\lambda)=\prod_{i=1}^m(1+v_i)\in \Z[K]$.

The realification of a complex vector bundle $\xi$ is denoted by $\xi_\R$.
Confusing notation we will denote by $\C$ and $\R$
a 1-dimensional trivial
complex or real  vector  bundle over a space $X$.
Now we can state our main theorem.

\begin{thm} \label{thm1}
Let $K$ be an finite simplicial complex over the vertex set $[m]$.
Then the following conditions are equivalent.
\nz
(i)
$K$ admits an $r$-coloring $K\lra \Delta[r]$.
\nz
(ii) The vector bundle $\lambda$
splits into a direct sum $ (\bigoplus_{i=1}^r \nu_i) \oplus \C^{m-r}$
of $r$ complex line bundles $\nu_i$
and a trivial $(m-r)$-dimensional complex bundle.
\nz
(iii) The realification $\lambda_\R$ of $\lambda$
splits into direct a sum $(\bigoplus_{i=1}^r \theta_i) \oplus \R^{2(m-r)}$
of $r$ $2$-dimensional
real bundles $\theta_i$ and a trivial $2(m-r)$-dimensional real bundle.
\nz
(iv) The vector bundle $\lambda$ is stably isomorphic to a direct sum $\bigoplus_{i=1}^r \nu_i$
of $r$ complex line bundles.
\nz
(v)The realification $\lambda_\R$ is stably isomorphic to a direct sum
$\bigoplus_{i=1}^r \theta_i$ of $r$ $2$-dimensional
real bundles.
\end{thm}

Several of our vector bundles  will be constructed
as homotopy orbit spaces.
For a compact Lie group $G$ and a $G$-space $X$,
the {\it Borel construction} or {\it homotopy orbit space}
$EG\times_G X$ will be denoted by $X_{hG}$. If $\eta\,\da\, X$ is an $n$-dimensional
$G$-vector bundle over $X$
with total space $E(\eta)$, the Borel construction establishes a fibre bundle
$E(\eta)_{hG} \lra X_{hG}$. In fact, this is an $n$-dimensional vector bundle over
$X_{hG}$, denoted by $\eta_{hG}$. For definitions and details see \cite{segal}

Let $M^{2n}$ be a quasitoric manifold over the simple polytope $P$. That is that
$M^{2n}$ carries a $T^n$-action, which is locally standard and
that the orbit space $M^{2n}/T^n=P$ is a
simple polytope. The Borel construction produces a space
$(M^{2n})_{hT^n}$, which is homotopy
equivalent to $DJ(K_P)$, where
$K_P$ is the simplicial complex dual to the boundary of $P$. For details see
\cite[Section 4.2]{daja}.
Let $\tau_M$ denote the tangent bundle of $M^{2n}$.
Davis and Januczkiewicz showed that the  vector bundle
$(\tau_M){hT^n}\da DJ(K)$ and $\lambda_{\R}$ are stably isomorphic
as real vector bundles over
$DJ(K_P)$ \cite[Section 6]{daja}.
We can draw the following corollary of Theorem \ref{thm1}

\begin{cor} \label{corollary} \label{equisplit}
Let $M^{2n}$ be a quasitoric manifold over a simple polytope $P$. Let $K_P$ be the
simplicial complex dual to the boundary of $P$.
If the tangent bundle $\tau_M$ of $M^{2n}$ is stably equivariantly isomorphic to a direct sum of $r$
2-dimensional equivariant $T^n$-bundles over $M^{2n}$, then $K_P$ admits an $r$-coloring.
\end{cor}

The paper is organized as follows.
For the proof of our main theorem we will need two different models for $DJ(K)$ . They are
discussed in the next section. In Section 3 we will use some geometric constructions to produce
a splitting of $\lambda$ from a given coloring.
The final section contains the proof of Theorem \ref{thm1}.

If not specified otherwise, $K$ will always denote an $(n-1)$-dimensional
finite simplicial complex with $m$-vertices.

We would like to thank Nigel Ray and Natalia Dobrinskaya for many helpful discussions.

\section{Models for $DJ(K)$} \label{models}

Let $\catk$ denote the category whose objects are the faces of $K$ and whose arrows are given by
the subset relations between the faces. $\catk$ has an initial object given by the empty face.
Given a pair $(X,Y)$ of pointed topological space we can define  covariant functors
$$
X^K,\ (X,Y)^K \colon \catk \lra \Top.
$$
The functor $X^K$ assigns to each face $\alpha$ the cartesian product
$X^\alpha$ and to each morphism $i_{\alpha,\beta}$ the inclusion
$X^\alpha \subset X^\beta$ where
missing coordinates are set to the base point $*$.
If $\alpha=\emptyset$, then $X^\alpha$ is a point.
And $(X,Y)^K$ assigns to
$\alpha$ the product
$X^\alpha\times Y^{[m]\setminus \alpha}$ and to $i_{\alpha,\beta}$
the coordinate
wise inclusion
$X^\alpha\times Y^{[m]\setminus \alpha}\subset
X^\beta\times Y^{[m]\setminus \beta}$.
The inclusions $X^\alpha \subset X^{[m]}=X^m$ and $X^\alpha\times Y^{[m]-\alpha}\subset X^m$
establish inclusions
$$
\colim_{\catk} X^K \lra X^m,\ \ \colim_{\catk} (X,Y)^K \lra X^m.
$$

We are interested in two particular cases, namely the functor $X^K$ for the
classifying space $BT=\C P^\infty$ of the 1-dimensional
circle $T$ and the functor
$(X,Y)^K$ for the pair $(D^2,S^1)$.
The colimit
$$
Z_K\letbe \colim_{\catk}\, (D^2,S^1)^K
$$
is called the {\it moment angle
complex} associated to $K$. The inclusions $Z_K \subset (D^2)^m\subset \C^m$
allow to restrict the standard $T^m$-action on $\C^m$  to $Z_K$.
The Borel construction produces a fibration
$$
q_K\colon (Z_K)_{hT^m} \lra BT^m
$$
with fiber $Z_K$. Moreover, $B_TK \letbe (Z_K)_{hT^m}$ is a realization of
the Stanley-Reisner algebra $\Z[K]$ and a model for $DJ(K)$. That is there exists an isomorphism
$H^*(B_TK;\Z)\cong \Z[K]$  such that the map $H^*(q_K;\bZ)$
can be identified with the map $\bZ[m]\lra \Z[K]$ \cite[Theorem 4.8]{daja}.
We will use this model for geometric construction with our vector bundles.

Buchstaber and Panov gave a different construction for $DJ(K)$.
They showed that
$c(K)\letbe \colim_{\catk}\, BT^K$ is
homotopy equivalent to $B_T K$ and that
the inclusion
$$
c(K) \lra  BT^m
$$
is homotopic to $q_K$ \cite[Theorem 6.29]{bupa}. In particular, each face $\alpha\in K$ defines a map
$h_\alpha\colon BT^\alpha \lra c(K)$.
The model $c(K)$ will be used to produce a coloring from a
stable splitting of $\lambda$.

\begin{rem}
If $K$ is the triangulation of an $(n-1)$-dimensional sphere, the moment angle complex $Z_K$ is a manifold.
In this case, the tangent bundle $\tau_Z$ is a $(m+n)$-dimensional $T^m$-equivariant vector bundle,
which satisfies the analogue of Corollary \ref{equisplit}.
If $\tau_Z$ is stably equivariantly isomorphic to a direct sum of $r$
2-dimensional equivariant $T^m$-bundles over $Z_K$, then $K$ admits an $r$ coloring.
Again this follows from the fact that $(\tau_Z)_{hT^m}$ and $\lambda_\R$ are stably isomorphic
\cite[Section 6]{daja}.
\end{rem}

\section{Geometric constructions} \label{xi}

The $m$-dimensional torus $T^m$ acts coordinate wise on $\C^m$. And the diagonal action
of $T^m$ on $\C^m\times Z_K$ makes the projection $\C^m\times Z_K \lra Z_K$ onto the second factor
into a $T^m$-equivariant complex vector bundle over $Z_K$, denoted by $\lambda'$.
An application of the Borel construction produces the bundle $\lambda\letbe\lambda'_{hT^m}\da B_TK$
over $B_TK$ whose total Chern class is given by
$c(\lambda)=\prod_i (1+v_i)\in \bZ[K]$ and whose classifying map is the composition
$B_TK\larrow{q_K} BT^m \lra BU(m)$.
Since $T^m$ acts coordinatewise on $\C^m$, both bundles, $\lambda'$ and $\lambda$ split into a direct sum of
(equivariant) line bundles. Let $\C_j$ denote the $j$-the component of $\C^m$.
in particular, $T^m$ acts on $\C_j$ via the projection
$T^m\lra T^{\{j\}}$ onto the $j$-th component of $T^m$. The vector bundle
$\lambda'_j\letbe \C_j\times Z_K$ is $T^m$-equivariant, and $\lambda_j\letbe (\lambda'_j)_{hT^m}$ is a
1-dimensional complex vector bundle over $B_TK$. We have
$\lambda'\cong \bigoplus_j \lambda'_j$ and $\lambda\cong \bigoplus_j \lambda_j$.
All this can be found in \cite[Section 6]{daja}.

If $g\colon K\lra \Delta[r]$ is an r-coloring we want to construct an equivariant splitting of
$\lambda'\da Z_K$ into a direct sum of $T^m$-equivariant complex line bundles and a trivial
bundle $\C^{m-r}$. We will use ideas of Davis and Januczkiewicz discussed in \cite[Section 6.2]{daja}.
For each $i\in [r]$ we denote by $S_i\letbe g^{-1}(i)\subset [m]$  the preimage of $i$ and
by $s_i\letbe |S_i|$ the order of $S_i$.
There are two vector bundles associated with $S_i$, namely
the tensor product $\nu_i\letbe \bigotimes_{j \in S_i} \lambda'_j$
of all complex line bundles associated to the vertices contained in $S_i$ and the direct sum
$\eta_i\letbe \bigoplus_{j\in S_i} \lambda'_j$
of all these line bundles. Both are
$T^m$-equivariant vector bundles over $\Z_K$.

\begin{lem} \label{epsilondelta}
For all $i\in [r]$,  there exists an $T^m$-equivariant vector bundle isomorphism
$\nu_i \oplus \C^{s_i-1} \lra \eta_i$.
\end{lem}

For simplicial complexes dual to the boundary of simple polytopes the claim is already stated in
\cite[Section 6.2]{daja}. We will give here a different proof.

\begin{proof}
For simplification we drop the subindex $i$ in the notation  and assume that $S=[s]$.
We will think of $\C^{s-1}\subset \C^s$ as the subspace given by
$\{(x_1,...,x_s)\in \C^s | \sum_k x_k=0\}$.
We define a map
$$
f: \C\times \C^{s-1} \times Z_K \lra \C^s\times Z_K
$$
by $f(y,x,z)\letbe (u,z)$ where the $j$-th coordinate $u_j$ of $u$ is given by
$u_j\letbe y\prod_{k\neq j, k\in [s]} \quer z_k + z_j x_j$. Here, $\quer z_k$ denotes
the complex conjugate of $z_k$.
If $T^m$ acts on $\C$ via the map $t\mapsto \prod_{j\in [s]} t_j$, trivially on
$\C^{s-1}$ and on $\C^s$ via the projection $T^m\lra T^s$ onto the first $s$ coordinates,
one can easily show that this map is $T^m$-equivariant. Moreover, with these actions,
the source is the total space of the bundle $\nu \oplus \C^{s-1}\da Z_K$
and the target the total space of $\eta\da Z_K$. Since both sides have the same dimension,
it is only left to show that $f$ is fiber wise a monomorphism.

By construction, any subset $\{j,k\}\subset [s]$ is a missing face in $K$.
Since $Z_K=\bigcup_{\alpha\in K} (D^2)^\alpha\times (S^1)^{[m]\setminus\alpha}$,
the space $(D^2)^{\{j,k\}}\times(S^1)^{[m]\setminus\{i,j\}}$
is not contained in  $Z_K$ and for $z=(z_1,...,z_m)\in Z_K$ there is at most one coordinate among $z_1,...,z_s$
which is trivial.

Now we assume that $f(y,x,z)=(0,z)$. In particular, we have $x_jz_j=-y\prod_{k\neq j} \quer z_k$.
If one of the coordinates $z_j$ vanishes, say $z_1=0$, then $z_j\neq 0$ for $j\neq 1$ and hence $y=0$
as well as $x_j=0$ for $j\neq 1$. Since $\sum_j x_j=0$, we also have $x_1=0$.

If $z_j\neq 0$ for all $j$, then $x_j=y\prod_{k\neq j} \quer z_k/z_j$ and
$0=\sum_j x_j=\sum_j y\prod_{k\neq j} \quer z_k/z_j=y\sum_j \prod_{k\neq j} \quer z_k/z_j$.
Multiplying with $\prod_j z_j$ shows that $y\sum_j\prod_{k\neq j} \quer z_kz_k=0$ and hence that
$y=0$ as well as $x_j=0$ for all $j$.
This shows that $f$ is a fiber wise monomorphism and finishes the proof.
\end{proof}

\begin{cor} \label{equisplits}
Let $K\lra  \Delta[r]$ be an $r$-coloring of a finite simplicial complex.
Then the following holds:
\nz
(i) The bundle $\lambda'\da Z_K$ splits equivariantly into a direct sum of $r$ equivariant complex
line  bundles  and a trivial bundle.
\nz
(ii) The bundle $\lambda\da DJ(K)$ splits into a direct sum of $r$ complex line bundles and a trivial bundle.
\end{cor}

\begin{proof}
By Proposition \ref{epsilondelta} we have
$$
\lambda'\cong \bigoplus_{j=1}^r \bigoplus_{i\in S_j} \lambda'_i\cong \bigoplus_{j=1}^r(\nu'_j\oplus \C^{s_i-1})
\cong (\bigoplus_{j=1}^r \nu'_j) \oplus \C^{m-r}.
$$
This proves the first part, the second follows from the first by applying the Borel construction.
\end{proof}

\section{Proof of Theorem \ref{thm1}}

The proof needs some preparation. For topological spaces $X$ and $Y$ we denote by $[X,Y]$
the set of homotopy classes of maps from $X$ to $Y$ and for two compact Lie groups
$G$ and $H$ by $\hom(H,G)$ the set of Lie group homomorphism $H\lra G$.

Let $G$ be a compact connected Lie group with maximal torus $j\colon\! T_G\hookrightarrow  G$
and Weyl group $W_G$.
Since the action of $W_G$ on $T_G$ is induced by conjugation with elements of $G$,
the composition of $w\in W_G$ and $j$ induces a map between  the classifying spaces
homotopic to $Bj$.
And passing to classifying spaces followed by composing with $Bj$ induces a map
$\hom(H,T_G) \lra [BH,BG]$ which factors through the orbit space of the $W_G$-action  on $\hom(H,T_G)$
and provides a map
$\hom(H,T_G)/W_G \lra [BH,BG]$.
The following two facts may be found in \cite{no1} and are needed for the proof of our main theorem.

\begin{thm} \cite{no1} \label{facts}
Let $G$ be a connected compact Lie group and $S$ a torus.
\nz
(i) The map $\hom(S,T_G)/W_G \lra [BS,BG]$ is a bijection.
\nz
(ii) The map $[BS,BG]\lra \Hom(H^*(BG;\Q),H^*(BS;\q))$ is an injection.
\end{thm}

The rational cohomology $H^*(BG;\q)\cong H^*(BT_G,\q)^{W_G}$ is the ring of polynomial invariants
of the induced $W_G$-action on the polynomial algebra $H^*(BT_G;\q)$.
For $G=SO(2k+1)$ the maximal torus $T_{SO(2k+1)}=T^k$ is an $k$-dimensional torus
and we can identify $H^*(BT_{SO(2k+1)});\z)$ with $\z[k]=\z[v_1,...,v_k]$. The Weyl group
$W_{SO(2k+1)}$ is the wreath product $\z/2\wr \Sigma_l$ where $(\z/2)^k$ acts on $T^k$ via
coordinate wise complex conjugation and $\Sigma_k$ via permutations of the coordinates.
The rational cohomology of $BSO(2k+1)$ is then  given by
$$
H^*(BSO(2k+1);\q)\cong \q[k]^{\z/2\wr \Sigma_k}\cong \q[p_1,...,p_k].
$$
The classes $p_i$ are already defined over $\Z$. On the one hand $p_i\in H^{4i}(BSO(2k+1);\Z)$ is the universal $i$-th
Pontrjagin class for oriented bundles and on the other hand
$p_i =(-1)^i\sigma_i(v_1^2,...,v_k^2)\in \z[k]^{\z/2\wr \Sigma_k}$ is
up to a sign the $i$-th elementary symmetric polynomial in the squares
of the generators of $\Z[k]$. In particular,
for an oriented $(2k+1)$-dimensional real vector bundle $\rho$ over a space $X$, the total
Pontrjagin class $p(\rho)=1+\sum_{i=1}^k p_i(\rho)$ determines completely the map
$H^*(BSO(2k+1);\q)\lra H^*(X;\q)$ induced by the classifying map $\rho\colon X \lra BSO(2k+1)$.

\begin{exa} \label{basicex}
Let $\rho \colon BT^s\lra BSO(2k+1)$ be the composition of a coordinate wise inclusion
$\hat\rho \colon BT^s \lra BT^k$ followed
by the maximal torus inclusion $BT^k\lra BSO(2k+1)$. Then the total Pontrjagin class of $\rho$ is given by
$p(\rho)=\prod_{i=1}^s (1-v_i^2)$, where we identify $H^*(BT^s;\z)$ with $\z[v_1,...v_s]$.

By Theorem \ref{facts}, up to homotopy every vector bundle $\omega \colon BT^s\lra BSO(2k+1)$
is the composition of a lift $\hat\omega \colon BT^s \lra BT^k$ and $Bj$. If
$p(\omega)=p(\rho)$ then both maps $\rho$ and $\omega$ are homotopic and the underlying homomorphisms
$j_\omega, j_\rho \colon T^s\lra T^k$
of the lifts
$\hat \omega$ and $\hat \rho$ differ only by an element of the Weyl group (Theorem \ref{facts}).
In particular, since $j_\rho$ is given by a coordinatewise inclusion, the homomorphism $j_\omega$
also is a coordinate wise inclusion  possibly followed by complex conjugation on some coordinates.
\end{exa}

{\it Proof of Theorem \ref{thm1}:}
If $K\lra \Delta[r]$ is an $r$-coloring, Corollary \ref{equisplits}
provides the appropriate splitting for $\lambda$.
The splitting conditions on $\lambda$ can be put into a hierarchy, a splitting of $\lambda$
establishes a splitting of $\lambda_\R$ and a stable isomorphism
between
$\lambda$ and a direct sum of $r$ complex line bundles, and the two latter conditions a stable
isomorphism between $\lambda_\R$ and a direct sum of $r$ $2$-dimensional real bundles.
It is only left to show that this last stable isomorphism allows to construct a coloring.

We will work with the model $c(K)$ for $DJ(K)$ and again describe vector bundles over $c(K)$
by their classifying maps. In particular, for $t\geq 2m$ the real vector bundle
$\rho_t\letbe \lambda_\R\oplus R^{t-2m}$ is a map $\rho_t: c(K) \lra BO(t)$.
Since we are considering stable splittings, we can pass from $\rho_t$ to $\rho_{t+1}$, if necessary,
and assume that $t=2s+1$ is odd. This will simplify the discussion. For example,
if $t$ is odd, we have $BO(t)\simeq BSO(t)\times B\z/2$. And since $c(K)$ is simply connected,
the bundle $\rho_t$ has a unique orientation given by
the first coordinate of the map $\rho_t\colon c(K)\lra BSO(t)\times B\z/2$. We also
denote this map by $\rho_t$. The total Pontrjagin class of $\rho_t$ is given by
$p(\rho_t)=\prod_{i=1}^m (1-v_i^2)$ \cite[Section 6]{daja}.

Let $\phi\colon BT^r\lra BSO(t)$ be the map induced by the composition
of the coordinate wise inclusion
$T^r\subset T^{s}$ into the first $r$ coordinates followed
by the maximal torus inclusion $T^{s}=T_{SO(t)}\subset SO(t)$.
A splitting $\rho_t\cong (\bigoplus_{j=1}^r \theta_j) \oplus \R^{t-2r}$ establishes a map
$\hat \rho_t\colon c(K)\lra BT^r$ such that $\phi \hat\rho_t\simeq \rho_t$.

Now let $\alpha\in K$ be a face and $h_\alpha\colon BT^\alpha \lra c(K)$ the associated map.
The composition $\hat\rho_t h_\alpha\colon  BT^\alpha \lra BT^r$ determines a unique homomorphism
$j_\alpha \colon T^\alpha \lra T^r$. The total Pontrjagin class of $\phi\hat\rho_t h_\alpha$ is given by
$p(\phi \hat\rho_t h_\alpha)=\prod_{i\in \alpha}(1-v_i^2)$. Example \ref{basicex} shows that
$j_\alpha$ is a coordinate wise inclusion $T^\alpha \lra T^k$ possibly followed by
complex conjugation on some coordinates. The coordinate wise inclusion defines an injection
$\alpha \lra [r]$. Since for any inclusion of faces $\beta \subset \alpha$ the restriction
$(\hat\rho_th_\alpha)|_{BT^\beta}$ equals the composition $\hat\rho_t h_\beta$, the underlying homomorphisms
satisfies the formula $j_\alpha|_\beta=j_\beta$. We can conclude that the collection of all these maps
defines a map $[m]\lra [r]$, whose restriction to any face of $K$ is an injection. This establishes a
non degenerate simplicial map $K\lra \Delta[r]$ which is an $r$-coloring for $K$. \qed

%
%
%
%
%
%
%
%
%
%
%


%
%
\end{document}